\documentclass[]{aspm}

\articleinfo{}{}{}

\usepackage{a4wide}
\usepackage{amsfonts}
\usepackage{amssymb}
\usepackage{amsmath}
\usepackage{amsthm}
\usepackage{color}

\usepackage{ifpdf}
\ifpdf
    \usepackage[pdftex]{graphics,graphicx}
    \DeclareGraphicsExtensions{.png,.pdf}
\else
   \usepackage[dvips]{graphics,graphicx}
   \DeclareGraphicsExtensions{.eps}
\fi
\graphicspath{{.}}

\newtheorem{theorem}{Theorem}
\newtheorem{lemma}{Lemma}
\newtheorem{corollary}{Corollary}
\newtheorem{definition}{Definition}
\newtheorem{remark}{Remark}

\def\p{\varphi}
\def\phi{\varphi}
\def\po{\p^\circ}
\def\Pp{P_\p}

\def\R{\mathbb{R}}
\def\Div{\mathrm{div}\,}
\def\e{\varepsilon}
\def\Wp{{W}_\p}
\def\Wpe{{W}_{\p_\e}}

\def\dpE{d_\p^E}

\def\npE{n_\p^E}
\def\wE{\widetilde{E}}
\def\wsto{\stackrel{*}{\rightharpoonup}}
\def\HH{\mathcal{H}}
\def\Lip{\textup{Lip}}
\DeclareMathOperator*{\essinf}{ess\,inf}
\DeclareMathOperator*{\esssup}{ess\,sup}

\begin{document}

\title[Crystalline curvature flow]{Existence and uniqueness for planar
anisotropic and crystalline curvature flow}

\author[Chambolle-Novaga]{A. Chambolle, M. Novaga}


\address{CMAP, Ecole Polytechnique, CNRS, 91128 Palaiseau, France,\\ 
{\tt antonin.chambolle@cmap.polytechnique.fr}
\\[2mm] 
Dipartimento di Matematica, Universit\`a di Pisa,\! Largo Pontecorvo 5,
56127 Pisa, Italy, 
\\{\tt novaga@dm.unipi.it}
}


\rcvdate{Month Day, Year}
\rvsdate{Month Day, Year}





\begin{abstract}
We prove short-time existence of $\p$-regular solutions to the planar anisotropic curvature flow, 
including the crystalline case, with 
an additional forcing term possibly unbounded and discontinuous in time, such as for instance a white noise. 
We also prove uniqueness of such solutions when the anisotropy is smooth and elliptic. 
The main tools are the use of an implicit variational scheme in order to define the evolution, 
and the approximation with flows corresponding to regular anisotropies.
\end{abstract}

\maketitle

\tableofcontents


\section{Introduction}

In this paper we consider the
anisotropic curvature flow of planar curves, corresponding to the evolution law
\begin{equation}\label{formalevol}
V\ =\ \kappa_\p + \frac{\partial G}{\partial t}
\end{equation}
in the Cahn-Hoffmann direction $n_\p$.
We shall assume that the forcing term  $G$ has the form
$G=G_1+G_2$ with $G_1,\,G_2$ satisfying:
\begin{enumerate}
\item[i)] $G_1\in C^0([0,\infty))$ does not depend on $x$;
\item[ii)] $G_2\in C^1([0,\infty);{\rm Lip}(\R^2))$.
\end{enumerate}

Observe that \eqref{formalevol} is only formal, as $\partial G_1/\partial t$ does not
necessarily exist, however the motion can still be defined
in an appropriate way (see Definition \ref{defpregflow}). Notice also that we
include the case of $G$ being a typical path of a Brownian
motion, which is necessary to take into account a stochastic forcing
term as in~\cite{DLN,LionsSouga}.
\smallskip

In the smooth anisotropic case, the
first existence and uniqueness results of a classical evolution in can be found in~\cite{An:90},
where S. Angenent showed existence, uniqueness and comparison for a
class of equations which include \eqref{formalevol} in the case $G=G_2$ and $\p$ regular.
The existence and uniqueness of a weak solution for the forced flow, with a Lipschitz continuous forcing term,
follows from standard viscosity theory \cite{UsersGuide,CGG}.

The crystalline curvature flow was mathematically formalized by
J.~Taylor in a series of papers (see for instance~\cite{Ta:78,Ta:93}). In two-dimensions, the existence of the flow
in the non forced case $G=0$ reduces to the analysis of a
system of ODEs. It was first shown by F.J. Almgren and J. Taylor in \cite{AT}, together with a proof of consistency of a variational scheme similar to the
one introduced in Section \ref{secATW}. The uniqueness and comparison
principle in the non forced case were established shortly after by
Y. Giga and M.E. Gurtin in \cite{GiGu:96}.
The forced crystalline flow was studied in~\cite{BGN:00}, however with
strong hypotheses on the forcing to ensure the preservation of the facets.

It is only in relatively recent work that the flow has been studied
with quite general forcing terms: in~\cite{GigaGigaRybka,GigaGorkaRybka} 
a Lipschitz forcing is considered. However, \cite{GigaGigaRybka} is
restricted to the evolution of graphs (although with a general mobility),
while \cite{GigaGorkaRybka} only considers rectangular anisotropies, and
assumes that the initial datum is close to the Wulff shape.
The paper~\cite{mcmf} deals with quite general forcing terms (roughly,
the same as in this paper), but requires the anisotropy to be smooth.
It shows the consistency of the variational scheme and a comparison
for regular evolutions. In~\cite{BCCN-vol}, the authors show the
existence of convex crystalline evolutions (extending their results
of \cite{BCCN}) with time-dependent (bounded) forcing terms and
apply it to show the existence of volume preserving flows.

We show here a general existence result for the two-dimensional crystalline
curvature flow, with ``natural'' mobility, which holds in two cases:
for a general forcing $G=G_1$ depending only on time, and for a regular
forcing $G=G_2$ with $\partial G_2/\partial t$ continuous in time and Lipschitz continuous in space.

Our proof relies on estimates for the variational scheme introduced
in~\cite{ATW,LuckhausSturz}, which show that, if the initial curve
has a strong regularity (expressed in terms of an internal and external
Wulff shape condition), then this regularity is preserved for some
time which depends only on the initial radius. Extending these proofs
to higher dimension would require quite strong regularity results
for nonlinear elliptic PDEs, which do not seem available at a first
glance. 

The paper is organized as follows: in the Section \ref{secreg} we 
define the ``anisotropy'' and introduce our notion of a ``regular''
curvature  flow for smooth and nonsmooth anisotropies. In Section~\ref{secATW}
we study the time-discrete implicit scheme of~\cite{ATW}, and
extend some regularity results of~\cite{BCCN} to the flow with forcing.
We then show in Section \ref{secsmooth} the main existence results, for
smooth anisotropies. The fundamental point is that the time of existence
only depends on the regularity of the initial curve. In the smooth
case, we also show uniqueness of regular evolutions. Eventually,
in Section~\ref{seccrystal} we extend the existence result to the crystalline
case, however this simply follows from an elementary approximation
lemma (Lemma~\ref{lemmapproxcrystal}), and the fact that the time of existence is uniformly
controlled in this approximation.

\section{$RW_\p$-condition and $\p$-regular flows}\label{secreg}

We call {\it anisotropy} a function $\p$ 
which is convex, one-homogeneous and coercive on $\R^2$.
We will also assume that $\p$ is even, i.e. $\p$ is a norm,
although we expect that the results of this paper
still hold in the general case (but some proofs become more
tedious to write).

We will always assume that there exists $c_0>0$ such that
\begin{equation}\label{coerbound}
c_0 |x|\ \le\ \p(x)\ \le \ c_0^{-1}|x|
\qquad \forall x\in\R^2.
\end{equation}
We denote by $\po$ the polar of $\p$, defined as
$$
\po(\nu) := \sup_{x:\,\p(x)\le 1} \nu\cdot x
\qquad \nu\in\R^2,
$$
it is obviously also a convex, one-homogeneous and even function on $\R^2$.
Notice that from \eqref{coerbound} it easily follows
$$
c_0 |\nu|\ \le\ \po(\nu)\ \le \ c_0^{-1}|\nu|
\qquad \forall \nu\in\R^2.
$$
We denote by $W_\p:=\{\p\le 1\}$ the unit ball 
of $\p$, which is usually called the {\it Wulff shape}.

We say that $\p$ is {\it smooth} if $\p\in C^2(\R^2\setminus\{0\})$ and $\p$ is {\it elliptic}
if $\p^2$ is strictly convex, that is $\nabla^2(\p^2)\ge c\,{\rm Id}$ in the distributional sense, for some $c>0$.
It is easy to check that $\p$ is smooth and elliptic iff $\po$ is smooth and elliptic. 

Given a set $E\subset\R^2$ we let $\dpE$ be the signed $\p$-distance function to $\partial E$
defined as
\[
\dpE(x) := \min_{y\in E} \p(x-y) -\min_{y\not\in E} \p(y-x)\,,
\]
We let $\nu_\p^E:=\nabla d^E_\p$
be the exterior $\p$-normal to $\partial E$, $n_\p\in \partial\po(\nu_\p)$
be the so-called {\it Cahn-Hoffmann} vector field (where $\partial$ denotes the subdifferential), 
and $\kappa_\p:={\rm div}n_\p$ be the $\p$-curvature of $\partial E$,
whenever they are defined. We also set $E^c:=\R^2\setminus E$.

Following \cite{BeNo:99} we give the following definition:

\begin{definition}[{\bf $RW_\p$-condition}]\label{defreg}
We say that a set $E$
satisfies the inner $R\Wp$-condition for some $R>0$ if 
\begin{equation}\label{innerballcond}
\overline E=\bigcup_{x:\, \dpE(x)\le -R} \left(x+R\Wp\right)
\end{equation}
and for any $r<R$ and $x\in\R^2$, $(x+r\Wp)\cap E^c$ is connected.

We say that $E$ satisfies the outer $R\,\Wp$-condition if 
its complementary $E^c$ satisfies the 
inner $R\,\Wp$-condition.

We say that $E$ satisfies the $R\,\Wp$-condition if it satisfies
both the inner and outer $R\,\Wp$-conditions.
\end{definition}

\begin{remark}\label{remando}\rm
Notice that, if $E$ satisfies the $R\,\Wp$-condition for some $R>0$, 
then $\partial E$ is locally a Lipschitz graph. Moreover, when $\p$ is smooth,
the $R\,\Wp$-condition implies that $\partial E$ is of class $C^{1,1}$ and
$\vert\kappa_\p\vert \le 1/R$ a.e.~on $\partial E$. 
In this case the connectedness condition in Definition \ref{defreg}
is automatically satisfied whenever \eqref{innerballcond} holds.
However, in the nonsmooth case one can have
some pathological examples if one removes the connectedness condition, as the one depicted in Fig.~\ref{BadCrystal}
when the Wulff shape is a square.
\end{remark}

\begin{remark}\label{ramengo}\rm
It is not difficult to show that $E$ satisfies the inner $R\,\Wp$-condition iff \eqref{innerballcond} holds and
the following property holds: for all $x$ such that $\dpE(x)=-R'>-R$, the set $\partial E\cap (x+R'\partial W_\p)$
is connected. 

By \eqref{innerballcond} it follows that $\partial E\cap (x+R'\partial W_\p)$ is either a segment (possibly a point)
or the union of two segments. In particular,
if $\p$ is elliptic, this is equivalent to say that there exists a unique point in $\partial E$
minimizing the $\p$-distance from $x$.
\end{remark}

\begin{figure}
\begin{center}
\includegraphics[height=5.0cm]{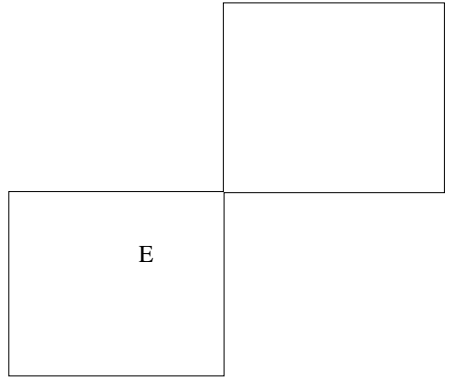}
\caption{\small{The pathological set described in Remark \ref{remando}.}}
\label{BadCrystal}
\end{center}
\end{figure}

\begin{definition}[{\bf $\p$-regular flows}]\label{defpregflow}
We say that a map $[0,T]\ni t\to \mathcal P(\R^2)$ defines a 
$\p$-regular flow for \eqref{formalevol} if 
\begin{enumerate}
\item $E(t)$ satisfies the $RW_\p$-condition for all $t\in [0,T]$ and for some $R>0$;
\item there exist an open set $U\subset\R^2$ 
and a vector field $z\in L^\infty([0,T]\times U;\R^2)$ such that 
\begin{enumerate}
\item $\partial E(t)\subset U$ for all $t\in [0,T]$,
\item $\dpE(t,x):=d^{E(t)}_\p(x)\in C^0([0,T];\Lip(U))$,
\item $z\in \partial\po(\nabla \dpE)$ a.e. in $[0,T]\times U$,
\item ${\rm div}z\in L^\infty([0,T]\times U)$;
\end{enumerate}
\item there exists $\lambda>0$ such that, for any $t,s$ with $0\le t<s \le T$ and a.e.~$x\in U$, there holds
\begin{equation} \label{smoothy}
\left|\dpE(s,x)-\dpE(t,x)
-\int_s^t\Div z(\tau,x)\,d\tau
-G(s,x)+G(t,x)\right|
\le \lambda(s-t)\max_{t\le \tau\le s}|\dpE(x,\tau)|\,.
\end{equation} 
\end{enumerate}
\end{definition}
\noindent Observe that~\eqref{smoothy} implies that $(d-G)$ is Lipschitz continuous,
so that \eqref{smoothy} can be rewritten as
\begin{equation}\label{smoothy2}
\left|\frac{\partial (\dpE-G)}{\partial t}(t,x)
-\Div\nabla \po(\nabla \dpE)(t,x)\right|\ \le\ \lambda|\dpE(t,x)|\,.
\end{equation}
for a.e.~$(t,x)\in [0,T]\times U$.
In case $G$ is $C^1$ in time, 
equation~\eqref{smoothy2} expresses the fact that 
$\partial E(t)$ evolves with speed given by \eqref{formalevol}.

\subsection{An approximation result.}

We now show that, given any set $E$ satisfying
the $R\Wp$-condition
for a general anisotropy $\p$,  
there exist smooth and elliptic
anisotropies $\p_\e\to\p$ and sets $E_\e\to E$, as $\e\to 0$, such that
$E_\e$ satisfies the $RW_{\p_\e}$-condition.
%
\begin{lemma}\label{lemmapproxcrystal}
Let $\p$ be a general
anisotropy and let $\p_\e$ be 
smooth and elliptic anisotropies converging to $\p$, with $\p_\e\ge \p$.
Let $E\subseteq\R^2$ satisfy the $R W_\p$-condition
for some $R>0$. Then there exist sets $E_\e$, with
$\partial E_\e\to \partial E$ as $\e\to 0$ in the Hausdorff sense, such that 
each $E_\e$ satisfies the $R\Wpe$-condition.
\end{lemma}

\begin{proof}
Let

\begin{eqnarray*}
\wE_\e &:=& \bigcup\left\{(x+RW_{\p_\e}):\,(x+RW_{\p_\e})\subset \overline E\right\}
\\
E_\e &:=& \R^2\setminus \bigcup\left\{ (x+RW_{\p_\e}):\,(x+RW_{\p_\e})\subset \overline{\wE_\e^c}\,\right\}.
\end{eqnarray*}
Notice that, by definition, $\wE_\e$ satsifies the innner $R\Wpe$-condition and $E_\e$ satisfies
the outer $R\Wpe$-condition, so that we have to prove that 
$E_\e$ also satisfies
the inner $R\Wpe$-condition.

\smallskip

\noindent {\it Step 1.} Let us show that $\partial E_\e\to \partial E$ as $\e\to 0$, in the Hausdorff sense.
In fact, this is obvious from the construction: since $\Wpe\subset\Wp$
and for any $x\in E$, there exists $z\in E$ with  $x\in z+R\Wp\subset E$,
we see that the distance from $x$ to $\wE_\e$ (and then $E_\e$)
is bounded  by the Hausdorff distance between $R\Wp$ and $R\Wpe$.
An estimate for the complement can be derived in the same way,
so that $d_\HH(\partial E_\e,\partial E)\le R d_\HH(\Wp,\Wpe)$.

%
%


\smallskip

\noindent {\it Step 2.} We now prove that 
$\wE_\e$ satisfies the outer $RW_{\p}$-condition.
We first show that, for all $x\in \partial\wE_\e$,
there exists $y$ such that 
\begin{equation}\label{equanime} 
\wE_\e\subset (y+RW_{\p})^c \qquad {\rm and}\qquad 
x\in \partial(y+RW_{\p})\,.
\end{equation}
Indeed, if $x\in \partial\wE_\e\cap \partial E$, \eqref{equanime} readily follows  
from the fact that $E$ satisfies the outer $RW_\p$-condition.

If $x\in \partial\wE_\e\setminus \partial E$,
then by definition of $\wE_\e$ there exists $x_1\in\R^2$ such that 
$$
(x_1+RW_{\p_\e})\subseteq \wE_\e \qquad {\rm and}\qquad x\in\partial (x_1+RW_{\p_\e})\,.
$$

Let $\ell_x$ be the maximal arc of $\partial (x_1+RW_{\p_\e})$ containing $x$ and contained in the interior of $E$,
and let $y_1,y_2\in\partial E$ be the endpoints of $\ell_x$. Notice that $\p_\e(y_1-y_2)<2R$.
Let  $y_3:=(y_1+y_2)/2$,
$R':=\p(y_1-y_2)/2\le \p_\e(y_1-y_2)/2<R$. 
As $E$ satisfies the inner $RW_\p$-condition, the set $(y_3+R'W_\p)$ has connected intersection with $E^c$, 
so that the set $(\partial E\setminus\partial \wE_\e)\cap (y_3+R'W_\p)$ 
contains a connected arc $\tilde\ell_x$ joining $y_1$ and $y_2$ (see Figure \ref{oneball}). 

Let $S_x$ be the subset of $E$ such that $\partial S_x=\ell_x\cup\tilde\ell_x$ and set 
$E':=\wE_\e\cap (y_3+R'W_\p)$. Notice that $S_x\subset E'$. As $E'$ is a convex set, there exists $y$ such that 
$x\in \partial (y+RW_\p)$ and $E'\subset (y+RW_\p)^c$. Moreover, since $E$ satisfies the outer $RW_\p$-condition, 
the set $E\cap \,{\rm int}(y+RW_\p)\supseteq S_x\cap \,{\rm int}(y+RW_\p)$ is connected. This implies that 
$\wE_\e\subset (y+RW_\p)^c$ and proves \eqref{equanime}.

In order to prove that $\wE_\e$ satisfies the outer $RW_\p$-condition, by Remark \ref{ramengo}
it remains to show that, given $\bar x$ with $d_\p^{\wE_\e}=R'<R$, the set  
$\partial \wE_\e\cap (\bar x+R'\partial W_\p)$ is connected. 

If $\partial \wE_\e\cap (\bar x+R'\partial W_\p)\subset\partial E$ this follows directly from the fact that 
$E$ satisfies the  outer $RW_\p$-condition. Otherwise, there exists 
$x\in (\partial \wE_\e\setminus \partial E)\cap (\bar x+R'\partial W_\p)$. In this case we claim that 
$\partial \wE_\e\cap (\bar x+R'\partial W_\p)=\{x\}$. Indeed, since  $\ell_x$ 
is a strictly convex arc, we have $\ell_x\cap (\bar x+R'\partial W_\p)=\{x\}$. Hence, if 
$\partial \wE_\e\cap (\bar x+R'\partial W_\p)$ contains another point $y\ne x$, then 
$y\not\in\overline{S_x}$. As $S_x\cap (\bar x+R' W_\p)\ne\emptyset$, it follows 
that $E\cap(\bar x+(R'+\delta)W_\p)$ contains at least
two connected components for $\delta>0$ sufficiently small, contradicting the fact that 
$E$ satisfies the  outer $RW_\p$-condition. Hence the set $(\partial E\setminus\partial \wE_\e)\cap (y_3+R'W_{\p_\e})$ 
is a connected arc $\tilde\ell_x$ joining $y_1$ and $y_2$. 

\smallskip 

\noindent {\it Step 3.} We prove that $E_\e$ satisfies the inner $R\Wpe$-condition by 
reasoning as in {\it Step 2}, with $E$ replaced by $(\wE_\e)^c$ (and $\p$ replaced by $\p_\e$). 
The only difference is due to the fact that 
$(\wE_\e)^c$ now satisfies inner $RW_{\p}$-condition and the outer $RW_{\p_\e}$-condition.
Therefore, letting $R':=\p_\e(y_1-y_2)/2<R$, the set $(y_3+R'W_{\p})\cap \wE_\e$
is connected, so that $(\partial \wE_\e\setminus\partial E_\e)\cap (y_3+R'W_{\p_\e})$ 
contains a connected arc joining $y_1$ and $y_2$. In the rest of the proof one can proceed as in {\it Step 2}.

\begin{figure}
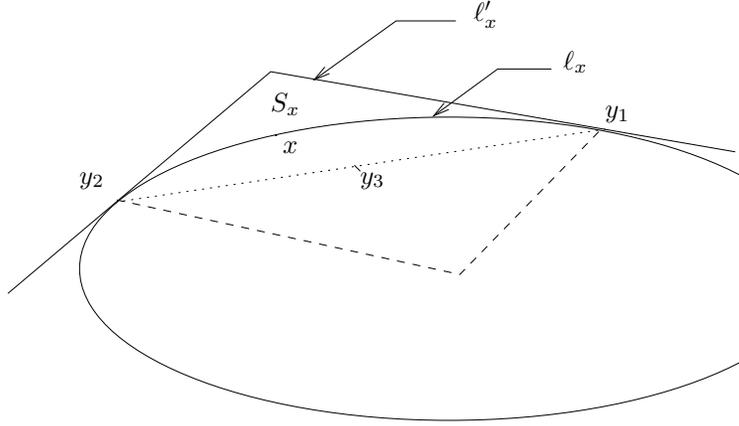

\begin{center}
\include{oneball}
\caption{\small{The configuration in Lemma \ref{lemmapproxcrystal}.}}
\label{oneball}
\end{center}
\end{figure}
\end{proof}

Lemma \ref{lemmapproxcrystal} has the following direct consequence.

\begin{corollary}
Let $E\subseteq\R^2$ satisfy
the $R W_\p$-condition
for some $R>0$. Then $E$ is {\it $\p$-regular} in the sense of \cite{BeNo:99}, that is,
there exists a
vector field $n_\p\in L^\infty(\{|d^E_\p|<R\},\R^2)$ 
such that $n_\p\in \partial\po(\nabla d^E_\p)$ a.e. in $\{|d^E_\p|<R\}$, and ${\rm div}n_\p\in L^\infty_{\rm loc}(\{|d^E_\p|<R\})$.
\end{corollary}

\begin{proof}
Take a sequence $\p_\e$ of 
smooth and elliptic anisotropies converging to $\p$, with $\p_\e\ge \p$.
By Lemma \ref{lemmapproxcrystal} we can approximate $E$ in the Hausdorff distance 
with sets $E_\e$ satisfyng the $RW_{\p_\e}$-condition. 
In particular, letting $n_{\p_\e}= \nabla\po_\e(\nabla d^{E_\e}_{\p_\e})\in L^\infty(\R^2)$ and recalling Remark \ref{remando}, we have
that ${\rm div}n_{\p_\e}\in L^\infty_{\rm loc}(\{|d^{E_\e}_{\p_\e}|<R\})$.
Therefore, any weak* limit $n_\p$ of $n_{\p_\e}$, as $\e\to 0$,
satisfies the thesis.
\end{proof}

\begin{remark}\rm
Notice that, given an arbitrary anisotropy $\p$, 
it is relatively easy to approximate it with
smooth and elliptic anisotropies $\p_\e$. For instance, one may let
$F_\e := \{\eta_\e*\po\le 1\}\oplus B(0,\e)$, 
with $\eta_r(x):= r^{-d} \eta\left(\frac x r\right)$,
and 
$\p_\e(x):=\sup_{\nu\in F_\e} \nu\cdot x$.
It is easy to check that the anisotropies $\p_\e$ are smooth and elliptic, and
converge locally uniformly to $\p$ as $\e\to 0$.
\end{remark}
\section{The time-discrete implicit scheme}\label{secATW}

The results of this section hold in any dimension $d\ge 2$
and are stated in this general form. Up to minor improvements,
they are essentially  stated in~\cite{mcmf,BCCN}.
Following~\cite{mcmf} we recall the definition and some properties of the implicit scheme introduced
in \cite{ATW,LuckhausSturz}. 
Given a set $E\subset \R^d$ with
compact boundary (we assume without loss of generality that
it is bounded), we define for $s>t\ge 0$ a transformation
$T_{t,s}$ by letting $T_{t, s}(E)=\{x\in B_R\,:\, w(x)<0\}$,
where $B_R=B(0,R)$, $R$ is large and $w$ is the minimizer
of 
\begin{equation}\label{eq:mainvp}
\min_{w\in L^2(B_R)} \int_{B_R}\po(D w)\ +\ \frac{1}{2(s-t)}
\int_{B_R} \Big(w(x)-\dpE(x)-G(s,x)+G(t,x)\Big)^2\,dx\,,
\end{equation}
whose existence and uniqueness is shown by standard methods.
One checks easily~\cite{AC-MCM,mcmf,ACN}
that for $R$ large, the level set $T_{t,s}(E)$ of $w$ does not depend on $R$,
and it is a solution to the variational problem
\begin{equation}\label{eq:mainvpset}
\min \Pp(F)
\,+\, \frac{1}{s-t}\int_{F} \left(\dpE(x)+G(s,x)-G(t,x)\right)dx\,,
\end{equation}
where the minimum is taken among the subsets $F$ of $\R^d$ with
finite perimeter, and we set
\[
\Pp(F):=\int_{\partial^* F} \po(\nu_F(x))d\HH^{1}(x).
\]
It follows that the set $T_{t,s}(E)$ has boundary of class 
$C^{1,\alpha}$, outside a compact singular set of
zero $\HH^{1}$-dimension~\cite{ATW}
(when $d=2$, the set $T_{t, s}(E)$ has boundary of class $C^{1,1}$).
The variational problem above is the generalization of the approach
proposed in~\cite{ATW,LuckhausSturz}, for building mean curvature flows
without driving terms, through an implicit time discretization.

For $s=t+h$, the Euler-Lagrange equation for $T_{t,t+h}(E)$ 
at a point $x\in \partial T_{t,t+h}(E)$ formally reads as
\[
\dpE(x) \,=\, -h\left(\kappa_\p(x)+\frac{G(t+h,x)-G(t,x)}{h}\right),
\]
with $\kappa_\p$ being 
the $\p$-curvature at $x$ of $\partial T_{t,t+h}(E)$, so
that it corresponds to an implicit time-discretization of~\eqref{formalevol}.
Observe also that this approximation is trivially
monotone: indeed if
$E\subseteq E'$ then $\dpE\geq d_\p^{E'}$, which yields $w\geq w'$,
$w$ and $w'$ being the solutions of~\eqref{eq:mainvp} for
the distance functions $\dpE$ and $d_\p^{E'}$ respectively.
We deduce that $\{w< 0\}\subseteq\{w'< 0\}$, that is,
\begin{equation}\label{inclusion}
E\subseteq E'
\ \Longrightarrow\ 
T_{t,t+h}(E)\subseteq T_{t,t+h} (E').
\end{equation}
\medskip

Consider now the Euler-Lagrange equation for~\eqref{eq:mainvp},
which is
\begin{equation}\label{ELmain}
-(s-t)\Div z \,+\,
w(x)\ =\ \dpE(x)+G(s,x)-G(t,x)
\end{equation}
for $x\in B_R$, with $\p(z(x))\le 1$ and
$z(x)\cdot \nabla w(x)=\po(\nabla w(x))$ a.e.~in $B_R$
(by elliptic regularity one knows that $w$ is Lipschitz).

We show that if $E$ is regular enough, then we have an estimate
on the quantity $\Div z + (G(s,x)-G(t,x))/(s-t)$ near the boundary of $E$.
The technique is adapted from~\cite{BCCN}.

\begin{lemma}\label{lemballsbis}
Assume that $E$ is a bounded set which satisfies the $\delta \mathcal{W}_\p$-condition
for some $\delta>0$.
Let $a<b$ be such that $X_{a,b} : = \{\max\{w,\dpE\}\geq a\}
\cap \{\min\{w,\dpE\}\leq b\} \subseteq \{\vert \dpE\vert <\delta\}$. Then
$\Div z \in L^\infty(X_{a,b})$ and
\begin{equation}\label{estima11}
\left\Vert \Div z + \frac{G(s,\cdot)-G(t,\cdot)}{s-t} \right\Vert_{L^\infty(X_{a,b})}
\ \leq\ 
\left\Vert \Div \npE  + \frac{G(s,\cdot)-G(t,\cdot)}{s-t}\right\Vert_{L^\infty(X_{a,b})}.
\end{equation}
\end{lemma}

\begin{proof}
Let $f:\R\to [0,+\infty)$ be a smooth increasing function
with $f(t)=0$ if $t\le 0$. Since
$(w,z)$ solves \eqref{ELmain}, we find
\begin{multline*}
  \int_{X_{a,b}} (w - \dpE)  f(w-\dpE)\,dx \ =\ 
\int_{X_{a,b}}\left((s-t) \Div z +  G(s,x)-G(t,x)\right)
 f(w-\dpE) \,dx
\\
 =\  (s-t)\int_{X_{a,b}} (\Div z - \Div \npE)  f(w-\dpE)\,dx \\+\,
\int_{X_{a,b}} \left((s-t)\Div \npE +  G(s,x)-G(t,x)\right) f(w-\dpE)\,dx
\ =:\  {\rm I} + {\rm II}.
\end{multline*}
We have, observing that $X_{a,b}$ has Lipschitz boundary (for a.e. choice
of $a,b$),
\begin{eqnarray*}
{\rm I} = & -& (s-t)\int_{X_{a,b}} (z - \npE) \cdot \nabla (w-\dpE)
f'(w-\dpE)\,dx
\\
& + & (s-t)\int_{\partial X_{a,b}} f(w-\dpE)(z - \npE)\cdot
\nu^{X_{a,b}} \,d \HH^{1} =: {\rm I}_1 + {\rm I}_2.
\end{eqnarray*}
First ${\rm I}_1\le 0$ since $z\cdot\nabla w=\po(\nabla w)$
and $z\cdot \nabla\dpE=\po(\nabla \dpE)$.
We claim that also ${\rm I}_2 \leq 0$. 
Indeed, on one hand, when $f(w-\dpE)>0$, $w>\dpE$ hence,
$\nu^{X_{a,b}} = \nu^{\{\dpE\leq b\}}=\nabla \dpE/|\nabla \dpE|$,
$\HH^{1}$-almost everywhere on $\{\min\{w,\dpE\} =  b\}$ while
$\nu^{X_{a,b}} = \nu^{\{w\geq a\}}=-\nabla w/|\nabla w|$,
$\HH^{1}$-almost everywhere on $\{\max\{w,\dpE\} = a\}$.
It follows that $f(w-\dpE)(z-\npE)\cdot\nu^{X_{a,b}}\le 0$
on both $\{\min\{w,\dpE\} =  b\}$ and $\{\max\{w,\dpE\} = a\}$
so that ${\rm I}_2\le 0$.
We conclude that ${\rm I} \leq 0$, hence
\begin{equation}\label{estima1}
\int_{X_{a,b}} (w - \dpE)f(w-\dpE)\,dx \ \le\   \int_{X_{a,b}} 
\left((s-t) \Div \npE +  G(s,x)-G(t,x)\right)\, f(w-\dpE)\,dx.
\end{equation}
Let $q>2$, let $r^+ := r \vee 0$, and let $\{f_n\}$ be a sequence
of smooth increasing nonnegative functions such that $f_n(r) \to
{r^+}^{(q-1)}$ uniformly as $n\to\infty$. {}From \eqref{estima1}
we obtain
\begin{multline*}
\int_{X_{a,b}} ((w - \dpE)^{+})^q \,dx \ \le\  \int_{X_{a,b}}
\left((s-t) \Div \npE +  G(s,x)-G(t,x)\right)\, ((w - \dpE)^{+})^{q-1} \,dx
\\ \le\  \int_{X_{a,b}}
\left((s-t) \Div \npE +  G(s,x)-G(t,x)\right)^+\, ((w - \dpE)^{+})^{q-1} \,dx.
\end{multline*}
Applying Young's inequality we obtain
$$
\Vert (w - \dpE)^{+} \Vert_{L^q(X_{a,b})}\ \leq\  
\Big\| \left((s-t) \Div \npE +
G(s,\cdot)-G(t,\cdot)\right)^+ \Big\|_{L^q(X_{a,b}\cap \{w>\dpE\})}\,.
$$
A similar proof, reverting the signs, shows that
$$
\Vert (w - \dpE)^{-} \Vert_{L^q(X_{a,b})} \ \leq\  
\Big\| \left((s-t) \Div \npE +
G(s,\cdot)-G(t,\cdot)\right)^- \Big\|_{L^q(X_{a,b}\cap \{w<\dpE\})}
$$
It follows that
$$
\Vert (s-t)\Div z+  G(s,\cdot)-G(t,\cdot) \Vert_{L^q(X_{a,b})}\ \leq
\Vert(s-t) \Div \npE +  G(s,\cdot)-G(t,\cdot) \Vert_{L^q(X_{a,b})}\,,
$$
and letting $q\to\infty$ we obtain \eqref{estima11}.
Observe that the estimate we may obtain is a bit more precise,
in fact we have shown:
\begin{multline}\label{estima11bis}
\essinf_{X_{a,b}\cap \{w<\dpE\}}
\Div \npE + \frac{G(s,\cdot)-G(t,\cdot)}{s-t}
\\ \le\ 
\Div z(x)  + \frac{G(s,x)-G(t,x)}{s-t}
\ \le\ 
\esssup_{X_{a,b}\cap \{w>\dpE\}}
\Div \npE + \frac{G(s,\cdot)-G(t,\cdot)}{s-t}
\end{multline}
for a.e. $x\in X_{a,b}$.
\end{proof}

We also recall Lemma~3.2 from~\cite{mcmf}:
\begin{lemma}\label{lemwulffh}
Let $x_0\in B_R$ and $\rho>0$, and let $t\geq 0$.
Let $\tilde w$ solve
\begin{equation}\label{vpwulff}
\min_{\tilde w\in L^2(B_R)} \int_{B_R}\po(D \tilde w)\ +\ \frac{1}{2h}
\int_{B_R} (\tilde w(x)-(\p(x-x_0)-\rho)-G(x,t+h)+G(x,t))^2\,dx\,.
\end{equation}
Then
\begin{equation}\label{superwulff}
\tilde w(x)\ \leq\ \begin{cases}
\displaystyle \p(x-x_0) + h\frac{1}{\p(x-x_0)} + \Delta_h(t) -\rho
& {\rm if}\ \p(x-x_0)\geq \sqrt{2h} \\
\displaystyle 2\sqrt{2h} + \Delta_h(t) -\rho & {\rm otherwise,}
\end{cases}
\end{equation}
where $\Delta_h(t):=\|G(\cdot,t+h)-G(\cdot,t)\|_{L^\infty(B_R)}$.
\end{lemma}
We deduce an estimate on $w-\dpE$, if $E$ has an inner $\rho\Wp$-condition:
indeed, in this case, if $h=s-t$,
\[
\dpE(x)\ \le\ \inf\left\{
\p(x-x_0)-\rho\,:\, \dpE(x_0)=-\rho \right\}
\]
with, in fact, equality in $\{ -\rho\le \dpE\le \rho'\}$, where $\rho'\ge 0$
is the radius of an outer $\rho'\Wp$-condition.
It follows from~\eqref{superwulff} that
\begin{equation}
w(x)\ \le\ \inf\left\{
\p(x-x_0) -\rho + h\frac{1}{\p(x-x_0)} + \Delta_h(t) 
:\, \dpE(x_0)=-\rho \right\}
\end{equation}
for $x$ with $\dpE(x)\ge -\rho+\sqrt{2h}$, and more precisely
if $\rho'\ge \dpE(x)\ge -\rho/2$,
\begin{equation}\label{estimwd}
w(x)\ \le\ \dpE(x) + \frac{2h}\rho + \Delta_h(t),
\end{equation}
as soon as $h\le \rho^2 / 16$.

\section{Smooth anisotropies}\label{secsmooth}

\subsection{Existence of $\p$-regular flows.}

We will prove, in dimension $d=2$, an existence result for the forced
curvature flow, first in case the anisotropy is smooth and elliptic.
For technical reason, we need the forcing term $G$ to be either
time-dependent only (case $G_2=0$), or smooth (globally Lipschitz
in space and time, case $G_1=0$).

\begin{theorem}\label{thexistsmooth}
Assume $G_1=0$ or $G_2=0$, and let $(\p,\po)$ be a smooth and elliptic
anisotropy and $E_0\subset \R^2$ an initial set
with compact boundary, satisfying both an $R\Wp$-internal and
external condition. Then, there exist $T>0$, 
and a $\p$-regular flow $E(t)$ defined on $[0,T]$ and starting
from $E(0)=E_0$.

More precisely, 
there exist $R'>0$ and a neighborhood $U$
of $\bigcup_{0\le t\le T} \partial E(t)$ in $\R^2$
such that the sets $E(t)$ satisfy the $R'W_\p$-condition for all $t\in [0,T]$,
the $\p$-signed distance function
$\dpE(t,x)$ from $\partial E(t)$ belongs to $C^0([0,T];\Lip(U))\cap L^\infty([0,T];C^{1,1}(U))$,
$(\dpE-G)\in \Lip([0,T]\times U)$ and  
\begin{equation}\label{smoothevol}
\left|\frac{\partial (\dpE-G)}{\partial t}(t,x)
-\Div\nabla \po(\nabla \dpE)(t,x)\right|\ \le\ \lambda|\dpE(t,x)|\,.
\end{equation}
for a.e.~$(t,x)\in [0,T]\times U$, where $\lambda$ is a positive constant.
Finally, the time $T$, the radius $R'$, the set $U$, 
and the constant $\lambda$ depend only on $R$ and $G$.
\end{theorem}



Theorem~\ref{thexistsmooth} will be proved by time-discretization.
Before, we need the following technical lemma.




\begin{lemma}\label{lemmWC}
Let $\p,\po$ be smooth and elliptic, and a set $E$ satisfy
a $R\,\Wp$-condition for some $R>0$. We also assume that
$E$ is simply connected ($\partial E$ is a $C^{1,1}$ Jordan curve).
Let $\delta\in (0,R)$ and consider a set $F$ (also simply connected),
such that $E_\delta\subset F\subset E^\delta$. Assume that
$\|\kappa_\p^F\|_{L^\infty(\partial F)}\le K$ for a constant $K<1/(2\delta)$.
Then $F$ has a $R'\Wp$-condition, with $R'=\min\{ R-\delta, (1-2\delta K)/K\}$.
\end{lemma}
\begin{proof}
We assume that $\partial F$ is at least $C^2$. If the result holds
in this case, then given a more general $C^{1,1}$ set we can smooth
it slightly, use the result for the approximations, and then pass to
the limit.

\noindent {\it Step 1.}
We have $E^\delta\setminus E_\delta=\bigcup_{x\in \partial E} x+\delta \Wp$,
and for any $x\in \partial E$, the set $x+\delta\Wp$ is tangent
to $\partial E_\delta$ (respectively, $\partial E^\delta$) at exactly
one point $x-\delta n_\p(x)$ (resp., $x+\delta n_\p$). We can
define $\Gamma^+_x$ and $\Gamma^-_x$ as the two arcs on $\partial(x+\delta\Wp)$
delimited by the points $x\pm\delta n_\p(x)$, the exponent $+$ and $-$
indicating that $\Gamma^\pm_x$ meets $\partial E$ right
``after'' or ``before'' $x$, relative to an arbitrarily chosen orientation
of the curve.

A first observation is that $\sharp(\partial F\cap \Gamma^\pm_x)=1$
for all $x$. Indeed, we check that this value is a continuous function
of $x$. If not, there will exist for instance a point where
$\sharp(\partial F\cap \Gamma^+_x)$ has a ``jump'', that is, where $\partial F$
is tangent to $\Gamma^+_x$ and contains a small piece of arc which is
inside $x+\delta\Wp$ and tangent to its boundary: in this case,
we deduce that $\kappa_\p^F(x)$ is larger than $1/\delta$ or less than
$-1/\delta$, a contradiction.

Since this value is continuous, it can only be odd (since 
$E_\delta\subset F\subset E^\delta$), moreover if it were larger
than $1$, there would be a connected component of $F$ 
(as well as one of its complement) in $E^\delta\setminus E_\delta$,
a contradiction.

\noindent{\it Step 2.} Let $\rho<\min\{ R-\delta, (1-2\delta K)/K\}$. Assume
that there exists $y\in F$ such that $W:=y+\rho\Wp\subset \overline{F}$ and
$y+\rho\Wp$ meets $\partial F$ in at least two points $z^-,z^+$
(with $z^+$ ``after'' $z^-$ with respect to the orientation
along $\partial E$). These points
must be isolated (otherwise there would be a point on $\partial F$
with curvature equal to $1/\rho>K$). Observe also that
$W\cap (E^\delta\setminus E_\delta)$ is connected (since $E^\delta$ has
an inner $(R-\delta)\Wp$-condition). To $z^+$, we can associate a
unique $x^+$ such that $z^+\in \Gamma^+_{x^+}$, and to $z^-$
a unique $x^-$ such that $z^-\in \Gamma^-_{x^-}$.
Then, the piece of curve $\Gamma$ of $\partial F$ between $z^-$ and $z^+$
lies in the region of $E^\delta\setminus E_\delta$ bounded by
$\Gamma^-_{x^-}$ and $\Gamma^+_{x^+}$, which contains points at
``distance'' at most $2\delta$ from $W$: more precisely,
$\Gamma\subset y+(\rho+2\delta)\Wp$. Hence, there exists $s\in (\rho,\rho+2\delta]$ such that $\Gamma$ is contained in $y+s\Wp$ and tangent to its boundary,
and thus a point of curvature larger than $1/s\ge 1/(\rho+2\delta)>K$
on $\partial \Wp$, which is a contradiction. Therefore, the Wulff
shapes $y+\rho\Wp$ which lie inside $F$ can touch its boundary at most
in one point, and an inner condition of radius $\min\{ R-\delta, (1-2\delta K)/K\}$
easily follows.

The proof of the outer condition is identical.
\end{proof}
\begin{remark}\label{remarkriro}
\textup{We can refine the lemma to consider a situation
where $E$ has an inner $R_i\Wp$-condition and a outer $R_o\Wp$-condition,
for two given radii $R_i,R_o>\delta$. We assume that
$-K_o\le \kappa_\p^F\le K_i$ for two nonnegative constants $K_i,K_o$, 
(still less than $1/(2\delta)$). It is then deduced that
$F$ has a inner $R'_i\Wp$-condition and a outer $R'_o\Wp$-condition,
with $R'_i=\min \{R_i-\delta,(1-2\delta K_i)/K_i\}$,
$R'_o=\min \{R_o-\delta,(1-2\delta K_o)/K_o\}$.}
\end{remark}

\begin{proof}[Proof of Theorem \ref{thexistsmooth}]

{}From~\eqref{estima11} and~\eqref{estima11bis},
we will obtain some regularity of the boundary of $T_{t+h,t}(E)$,
which will allow to iterate the variational scheme. To simplify
(and without loss of generality) we assume that the initial
curve is a Jordan curve ($E_0$ is simply connected). If not,
one may evolve separately each connected component of the boundary.

\noindent{\it Step 1.a.: The case $G_1=0$.}
In the case $G_1=0$, there exists $C$ such that $\Delta_h(t)\le Ch$.
It follows from \eqref{estimwd} that if $E$ satisfies
the $\rho\Wp$-condition, then, by solving~\eqref{eq:mainvp} with
$s=t+h$,
\begin{equation}\label{estimdistance}
|w(x)-\dpE(x)|\ \le\ h\left(C+\frac{2}{\rho}\right)
\end{equation}
in $\{|\dpE|< \rho/2\}$. By standard comparison (using for
instance Lemma~\ref{lemwulffh} again) one also can check that $w<0$
if $\dpE\le\rho/2$, and $w>0$ if $\dpE\ge\rho/2$, so that
the boundary of $T_{t,t+h}(E)$ is at ($\p$-)distance
of order $h$ from $\partial E$, if $h\le \rho^2/36$ (Lemma~\ref{lemwulffh}).
We also observe that the Hausdorff distance between the
sets $E$ and $T_{t,t+h}(E)$ is of the same order, or equivalently,
$\|\dpE-d_\p^{T_{t,t+h}(E)}\|_{L^\infty(\R^d)}\le (C+2/\rho)h$.

A further observation is that if $E$ is simply
connected, also $T_{t,t+h}(E)$ is. Indeed, if not, there would
be a connected component of either $T_{t,t+h}(E)$ or its complement
in the set $\{|\dpE|\le h(C+2/\rho)\}$. Assume $F$ is a connected
of $T_{t,t+h}(E)$ wich lies in $\{|\dpE|\le h(C+2/\rho)\}$, so that
$|F|\le 2h\Pp(E)(C+2/\rho)$. One has that (using the isoperimetric
inequality)
\begin{multline*}
\Pp(F)+\frac{1}{h}\int_F \dpE(x)+G(t+h,x)-G(t,x)\,dx
\\ \ge \ 2\sqrt{|\Wp||F|}- 2|F|\left(C+\frac{1}{\rho}\right)
\\ \ge\ 2\sqrt{|F|}
\left(\sqrt{|\Wp|}- \sqrt{|F|}\left(C+\frac{1}{\rho}\right)\right)
\end{multline*}
which is positive if $h$ is small enough (depending on $C,\rho,\Pp(E)$),
showing that $T_{t,t+h}(E)\setminus F$ has an energy strictly lower
than $T_{t,t+h}(E)$ in~\eqref{eq:mainvpset}, a contradiction.\smallskip

Sending both $a$ and $b$ to $0$, one deduces from  \eqref{estima11}
that $T_{t,t+h}(E)$ has $C^{1,1}$ boundary, and moreover
\begin{multline*}
\left\| \Div n_\p^{T_{t,t+h}(E)} + \frac{1}{h}(G(t+h,\cdot)-G(t,\cdot)) \right\|
_{L^\infty(\partial {T_{t,t+h}(E)})}
\\ 
\leq\ 
\left\| \Div \npE  + \frac{1}{h}(G(t+h,\cdot)-G(t,\cdot))\right\|_{L^\infty(E\triangle T_{t,t+h}(E))}.
\end{multline*}
Since $(G(t+h,x)-G(t,x))/h$ is $L$-Lipschitz in $x$ for some $L>0$, and
$\Div \npE$ is bounded by $4/\rho^2$ in $\{ |\dpE|\le \rho/2\}$, we deduce
\begin{multline}\label{totalcontrol}
\left\| \Div n_\p^{T_{t,t+h}(E)} + \frac{1}{h}(G(t+h,\cdot)-G(t,\cdot))
\right\|_{L^\infty(\partial {T_{t,t+h}(E)})}
\\ \leq\ 
\left\| \Div \npE  + \frac{1}{h}(G(t+h,\cdot)-G(t,\cdot))\right\|_{L^\infty(\partial E)}\ +\ \frac{4h}{\rho^2}\left(C+\frac{2}{\rho}\right)
\\ \leq\ 
\left\| \Div \npE  + \frac{1}{h}(G(t,\cdot)-G(t-h,\cdot))\right\|_{L^\infty(\partial E)}\ +\ h\left(L+ \frac{4}{\rho^2}\right)\left(C+\frac{2}{\rho}\right)\,,
\end{multline}
provided $h$ is small enough (depending on $\rho,L,C$).
Eventually,  it follows that the curvature of $\partial T_{t,t+h}(E)$
 (since $d=2$, the total and mean curvature coincide) has a global
estimate $1/\rho+2C+O(h)$, and one will
deduce from Lemma~\ref{lemmWC}
that for $h$ small enough, this new set
also satisfies the $\rho'\Wp$-condition, with
$\rho'=\rho/(1+(2C+O(h))\rho )>0$, provided the assumptions of the lemma
are fulfilled. 

We now consider $E_0,R$ as in Theorem~\ref{thexistsmooth}, and
let for $h>0$ and any $n\ge 1$, $E^h_n = T_{(n-1)h,nh}(E_0)$.
We also define $E^h(t)=E^h_{[t/h]}$ for $t\ge 0$.
A first observation is that if $x\in (E_0)_R$, $x+R\Wp\subset E_0$ 
so that if $r(t)$ solves $\dot{r}=-(1/r+C)$ with $r(0)=R$,
for any $\eta>0$ (small), $x+(r(t)-\eta)\Wp \subset E^h(t)$ for
$h$ small enough, as long as $r(t)\ge\eta$.
The function $r(t)$ solves $r(t)-R-\ln\left(\frac{1+Cr(t)}{1+CR}\right)/C=-Ct$,
and given $\delta \in (0,R)$ (which will be precised later on),
there exists $T_1(R,C,\delta)$ such that
if $t\le T_1$ and $h>0$ is small enough,
\begin{equation}\label{subsetU}
(E_0)_\delta\ \subset\ E^h(t)\ \subset\ (E_0)^\delta\,.
\end{equation}
We let $U=\{|\dpE|\le \delta\}$.

Letting $E^h_1= T_{0,h}(E_0)$, we deduce from~\eqref{totalcontrol}
that if $h<R^2/36$ is small enough,
\[
A_1^h\ :=\ \left\|\Div n_\p^{E^h_1}+ \frac{1}{h} (G(h,\cdot)-G(0,\cdot))
\right\|_{L^\infty(\partial {E^h_1})}\ \le\ \frac{1}{R}+C+ 
\frac{1}{4}\left(C+\frac{2}{R}\right)\ =:\ M_1.
\]
For $n\ge 1$, we then
define iteratively the sets $E^h_{n+1}=T_{nh,(n+1)h} (E^h_n)$ and let
\[
A_{n+1}^h\ :=\ \left\|\Div n_\p^{E^h_{n+1}}
+ \frac{1}{h} (G((n+1)h,\cdot)-G(nh,\cdot))\right\|_{L^\infty(\partial {E^h_{n+1}})}\,.
\]
Let now $R_1 = (2M_1+C))^{-1}$.
As long as $A_n\ge 2M_1$, one can deduce from Lemma~\ref{lemmWC},
using~\eqref{subsetU} and provided we had chosen $\delta<R_1/2$,
that $E_{n+1}^h$ satisfies
the $R_1\Wp$-condition, so that~\eqref{totalcontrol} holds
(with $E=E^h_n$, $\rho=R_1$) and
\[
A_{n+1}^h\ \le\ A_n^h \,+\, 
h\left(L+ \frac{4}{R_1^2}\right)\left(C+\frac{2}{R_1}\right)\,.
\]
By induction, we deduce that (letting $B=(L+4/R_1^2)(C+2/R_1)$)
$A_{n+1}^h\le M_1+(n+1)hB$ as long as $nh\le \min\{T_1,M_1/B\}:=T>0$.

We observe that since $\delta<R_1/2$, as long
as $nh\le T$, not only $\partial E_n^h\subset U$,
but all the signed distance functions to the boundaries of $E_n^h$
are in $C^{1,1}(U)$. Notice that $T$ and the width $\delta$
of the strip $U$ depend only on $R,C,L$.

\noindent{\it Step 1.b.: The case $G_2=0$.}
We now show that we can obtain a similar control in case of a space
independent forcing term, which can be the derivative of a 
continuous function $G_2$ (a relevant example is a Brownian forcing). 
In that case, we can consider the algorithm from
a different point of view: given the set $E$, we first consider
the set $E'$ with signed distance function $d_\p^{E'}:=\dpE(x)+G(s)-G(t)$,
then, we apply to this set $E'$ the algorithm
with $G\equiv 0$, that is, we solve~\eqref{eq:mainvp} for $E=E'$ and $G=0$:
\begin{equation*}
\min_{w\in L^2(B_R)} \int_{B_R}\po(D w)\ +\ \frac{1}{2(s-t)}
\int_{B_R} \Big(w(x)-d_\p^{E'}(x)\Big)^2\,dx\,,
\end{equation*}
and then let $E''=\{w<0\}$. It is clear that this is equivalent to the
original algorithm, so that $E''=T_{t,s}(E)$.

Assume in addition that $E$ has an inner $r_i\Wp$-condition  and
a outer $r_o\Wp$ condition, for some radii $r_i,r_o>0$. If $(s-t)$ is
small enough,
then $E'$ has the
inner $r'_i\Wp$-condition  and outer $r'_o\Wp$ condition with
$r'_i=r_i-G(s)+G(t)$ and $r'_o=r_o+G(s)-G(t)$. In particular,
$d_\p^{E'}=\dpE(x)+G(s)-G(t)$ is locally $C^{1,1}$ in the strip
$\{ -r'_i<d_\p^{E'}<r'_o\}$ and the surface $\partial E'$ has a
curvature which satisfies a.e.
\begin{equation}\label{innerouterbound}
-\frac{1}{r'_o}\ \le\ \Div n_\p^{E'}\ \le\ \frac{1}{r'_i}\,.
\end{equation}
As before, from \eqref{estimwd} we have that, if $h=s-t$ is small
enough, then
\begin{equation}\label{estimdistance2}
|w(x)-d_\p^{E'}(x)|\ \le\ \frac{2h}{\min\{r'_i,r'_o\}}\,,
\end{equation}
showing that the boundary of $T_{t,t+h}(E)$ remains close
to the boundary of $E'$ (provided $r'_i,r'_o$ are controlled
from below).

From~\eqref{estima11bis} (with $G=0,E=E'$)
and~\eqref{innerouterbound}, \eqref{estimdistance2}, we
obtain that if $h$ is small enough,
\begin{equation*}
-\frac{1}{r'_o}- \frac{2h}{(r'_o)^2\min\{r'_i,r'_o\}}\ \le
\ \Div n_\p^{T_{t,t+h}(E)}\ \le
\ \frac{1}{r'_i}+ \frac{2h}{(r'_i)^2\min\{r'_i,r'_o\}}\,,
\end{equation*}
and in particular we can deduce from 
Lemma~\ref{lemmWC} and Remark~\ref{remarkriro} that 
$T_h(E)$ satisfies the inner $r''_i\Wp$
and  outer $r''_o\Wp$-conditions with
\begin{equation*}
r''_i \ \ge\ r'_i - \frac{ch}{r'_i}\,,
\qquad r''_o \ \ge\ r'_o - \frac{ch}{r'_o}\,,
\end{equation*}
for some constant $c>0$.


As in the previous step, we
now consider $E_0,R$ as in Theorem~\ref{thexistsmooth}, we
let $E_0^h=E_0$
and define for each $n\ge 0$, $E_{n+1}^h:= T_{nh,(n+1)h}(E_n^h)$.
Let $r_o^0=r_i^0=R$. The previous analysis shows that $E_1^h$
has the inner $r_i^1\Wp$ and the outer $r_o^1\Wp$-conditions with
\[
r_i^1 \ \ge\ r_i^0 -G(h)+G(0) - \frac{ch}{R},
\qquad
r_o^1 \ \ge\ r_o^0 +G(h)-G(0) - \frac{ch}{R},
\]
provided $|G(h)-G(0)|\le R/2$ (for some constant $c>0$).
Now, assuming that $n$ is such that
\[
r_i^n \ \ge\ r_i^0 -G(nh)+G(0) - \frac{c nh}{R},
\qquad
r_o^n \ \ge\ r_o^0 +G(nh)-G(0) - \frac{c nh}{R},
\]
we deduce that
\[
r_i^{n+1} \ \ge\ r_i^0 -G((n+1)h)+G(0) - \frac{c(n+1)h}{R},
\ 
r_o^{n+1} \ \ge\ r_o^0 +G((n+1)h)-G(0) - \frac{c(n+1)h}{R},
\]
as long as $|G((n+1)h)-G(0)|+ c(n+1)h/R \le R/2$.
Define $T$ such that $\max_{0\le t\le T} |G(t)-G(0)|+ct/R\le R/4$,
and let $U=\{ |d_\p^{E_0}|<R/4\}$: then, on one hand,
$\partial E_n^h\subset U$ for all $n\ge 0$ with $nh\le T$, on
the other hand, $E_n^h$ satiafies the $(R/2)\Wp$-condition, 
so that $d_\p^{E_n^h}\in C^{1,1}(U)$. Again, $U$ and $T$ depend
only on $G$ and $R$.

\noindent{\it Step 3: Conclusion.} For $t\in [0,T]$ and $h$ small, we let
$E_h(t) = E^h_{[t/h]}$, $d_h(t,x)=d_\p^{E(t)}(x)$, and we now  send $h\to 0$. 
Since $d_h-G$ is uniformly Lipschitz in $[0,T]\times U$ (in 
time, in fact, we have $|d_h(t,x)-G(t,x)-d_h(s,x)+G(s,x)|\le c|t-s|$ if
$|t-s|\ge h$, for some constant $c$), up to a subsequence $(h_k)$ it converges
uniformly to some $d$ with $d-G\in \Lip([0,T]\times U)$, moreover,
at each $t>0$, $E_{h_k}(t)$ converges (Hausdorff) to a set $E(t)$ with
$d(t,x)=d_\p^{E(t)}(x)$. Let us establish~\eqref{smoothevol}.

For $n\le T/h-1$ and $x\in \partial E^h_{n+1}$, by definition
of the scheme we have 
\[
-d_\p^{E^h_n}(x) - h\,\Div n_\p^{E^h_{n+1}}(x) - G(t+h,x)+G(t,x)\ =\ 0.
\] 
As $(G(t+h,\cdot)-G(t,\cdot))/h$
is $L$-Lipschitz in $U$, there holds
\[
\left|(G(t+h,x)-G(t,x))-\left(G(t+h,\Pi_{\partial E^h_{n+1}}(x))-G(t,\Pi_{\partial E^h_{n+1}}(x))\right)\right|
\le C h|d_\p^{E^h_{n+1}}(x)|
\]
where $C$ depends only on $L$ and $\p$, where we set
$$
\Pi_{\partial E^h_{n+1}}(x) = x - d^{E^h_{n+1}}_\p(x)n^{E^h_{n+1}}_\p(x).
$$
Choose now $x\in U$  such that $d_\p^{E^h_{n+1}}(x)\ge 0$.
In this case, it follows that
\[
d_\p^{E^h_n}(x)-d_\p^{E^h_n}(\Pi_{\partial E^h_{n+1}}(x))
\ \le\ \p(x-\Pi_{\partial E^h_{n+1}}(x))\ =\ d_\p^{E^h_{n+1}}(x).
\]
Hence, 
\begin{align*}
& d_\p^{E^h_{n+1}}(x)
-d_\p^{E^h_n}(x) - h\Div n_\p^{E^h_{n+1}}(x)
\\ 
& \ge -d_\p^{E^h_n}(\Pi_{\partial E^h_{n+1}}(x))
- h\Div n_\p^{E^h_{n+1}}(\Pi_{\partial E^h_{n+1}}(x)) + O\left( h|d_\p^{E^h_{n+1}}(x)|\right)
\\
&= G(t+h,\Pi_{\partial E^h_{n+1}}(x))-G(t,\Pi_{\partial E^h_{n+1}}(x)) + O\left( h|d_\p^{E^h_{n+1}}(x)|\right)
\\
&= G(t+h,x)-G(t,x)+ O\left( h|d_\p^{E^h_{n+1}}(x)|\right).
\end{align*}
Dividing by $h$ and letting $h\to 0^+$, we then get
\[
\frac{\partial (\dpE-G)}{\partial t}(t,x)
-\Div\nabla \po(\nabla \dpE)(t,x)\ \ge\ O\left(|\dpE(t,x)|\right)
\qquad (t,x)\in U\times [0,T]\cap \{\dpE(t,x)>0\},
\]
which implies 
\[
\frac{\partial (\dpE-G)}{\partial t}(t,x)
-\Div\nabla \po(\nabla \dpE)(t,x)\ \ge\ O\left(|\dpE(t,x)|\right)
\qquad (t,x)\in U\times [0,T].
\]
By taking $x\in U$ such that $d_\p^{E^h_{n+1}}(x)\le 0$, reasoning as above we get
\[
\frac{\partial \dpE}{\partial t}(t,x)
-\Div\nabla \po(\nabla \dpE)(t,x)-g(t,x)\ \le\ O\left(|\dpE(t,x)|\right)
\qquad (t,x)\in U\times [0,T],
\]
thus obtaining \eqref{smoothevol}.
\end{proof}


\begin{remark}\textup{
When $\p(x)=|x|$ and $G_2=0$, the an existence and uniqueness result for $\p$-regular flows
has been proved in \cite{DLN} in any dimension.}
\end{remark}

\subsection{Uniqueness of $\p$-regular flows.}\label{secuniq}

We now show uniqueness of the regular evolutions given
by Theorem~\ref{thexistsmooth}. 

\begin{theorem}\label{propuniq}
Given an initial set $E_0$, the flow of Theorem~\ref{thexistsmooth}
is unique. More precisely, if two flows $E$, $E'$ are given, starting
from initial sets $E_0\subseteq E'_0$, then $E(t)\subseteq E(t)$ for
all $t\in [0,\min\{T,T'\}]$ (where $T,T'$ are respectively the time
of existence of regular flows starting from $E_0$ and $E'_0$).
\end{theorem}

The thesis essentially follows from the results in~\cite{mcmf}. 
Indeed, in~\cite{mcmf} it is proved 
a comparison result for strict $C^2$ sub- and superflows, 
based again on a consistency result for the scheme defined in
Section~\ref{secATW}. A strict $C^2$ subflow is defined a in 
Theorem~\ref{thexistsmooth}, except that $\dpE(t,x)$ is
required to be in $C^0([0,T];C^2(U))$, and \eqref{smoothevol} is replaced
with (for $0\le t<s\le T$, $x\in U$)
\begin{equation}\label{smoothsubflow}
\dpE(s,x)-\dpE(t,x)
-\int_s^t\Div\nabla \po(\nabla \dpE)(\tau,x)\,d\tau
-G(s,x)+G(t,x)
\ \le\ -\delta(s-t) 
\end{equation}
for some $\delta>0$. A superflow will satisfy the reverse inequality,
with $-\delta(s-t)$ replaced with $\delta(s-t)$.
For technical reasons (in order
to make sure, in fact, that the duration time of these flows is
independent on $\delta$), we will ask that these flows are defined,
in fact,
in a tubular neighborhood $W$ of $\bigcup_{0\le t \le T}\partial E(t)$,
not necessarily of the form $[0,T]\times U$.

The thesis then follows from the consistency result
in~\cite[Thm.~3.3]{mcmf}, once we show the following approximation result.

\begin{lemma}
Let $E(t)$ be an evolution as in Theorem~\ref{thexistsmooth},
starting from a compact set $E_0$ satisfying the $R\Wp$-conditions for some $R>0$.
Then, for any $\e>0$, there exists $T'>0$ (depending only on $R$),
a set $E'_0$
and a strict $C^2$ subflow $E'(t)$ starting from $E'_0$ such that for all
$t\in [0,T']$, $E(t)\subset\subset E'(t)\subset \{\dpE\le \e\}$.
\end{lemma}

\begin{proof}
We sketch the proof and refer to~\cite{ACN} for more details.

The idea is to let first $d^\alpha=\dpE-\alpha t - \alpha/(4\lambda)$, for some small $\alpha>0$, 
with $\alpha (T+1/(4\lambda))< \e$.
One can then deduce from~\eqref{smoothevol} that, for all $s>t$,
\begin{multline*}
d^\alpha(s,x)-d^\alpha(t,x)
-\int_s^t\Div\nabla \po(\nabla d^\alpha)(\tau,x)\,d\tau
-G(s,x)+G(t,x)
\\ \le\ (s-t)(\lambda\max_{t\le \tau\le s}|\dpE(x,\tau)|-\alpha)+
\ \le\ \lambda(s-t)\left(\max_{t\le \tau\le s}|d^\alpha(x,\tau)|
+\alpha(s-{\textstyle\frac{3}{4}}\lambda^{-1})\right)\,.
\end{multline*}
Let $T':=\min\{T,1/(2\lambda)\}$, and let $\beta=\alpha/(8\lambda)$:
then if we let $W=\{(t,x)\,:\,0\le t\le T\,, |d^\alpha(t,x)|<\beta\}$,
we deduce that for $(t,x), (s,x)\in W$,
\[
d^\alpha(s,x)-d^\alpha(t,x)
-\int_s^t\Div\nabla \po(\nabla d^\alpha)(\tau,x)\,d\tau
-G(s,x)+G(t,x) \ \le\ -{\beta}(t-s)\,.
\]
Hence $\{d^\alpha\le 0\}$ is almost a $C^2$ subflow, except for the fact that it is
not $C^2$. However, this is not really an issue, as we will now check.
Consider indeed a spatial mollifier 
\begin{equation}\label{eqmoll}
\eta_r(x)= \frac{1}{r^d} \eta\left(\frac x r\right)
\end{equation}
where as usual $\eta\in C_c^\infty(B(0,1);\R_+)$, $\int_{\R^d}\eta(x)\,dx=1$.
Let $d^\alpha_r=\eta_r*d^\alpha$ (for $r$ small).
Observing, as before, that $G(s,x)-G(t,x)$ is $(s-t)L$-Lipschitz in $x$, 
one has
$|\eta_r*(G(s,\cdot)-G(t,\cdot))(x)-(G(s,x)-G(t,x))|\le (s-t)Lr$.
Hence, the level set $0$ of $d^\alpha_r$ will be a strict $C^2$ subflow,
for $r$ small enough,
if we can check that the difference
\begin{equation}\label{errormollif}
\eta_r * (\Div\nabla \po(\nabla d^\alpha)(\tau,\cdot) )(x)
\,-\, \Div\nabla \po(\eta_r*\nabla d^\alpha)(\tau,x) 
\end{equation}
can be made arbitrarily small for $r$ small enough and
any $(\tau,x)\in W$ (possibly reducing slightly the
width of $W$).
Now, for $(\tau,x)\in W$,
\[
\eta_r * (\Div\nabla \po(\nabla d^\alpha)(\tau,\cdot) )(x)
\ =\ \int_{B(0,r)} \eta_r(z) D^2\po(\nabla d^\alpha(\tau,x-z)):D^2 d^\alpha(\tau,x-z)\,dz
\]
while
\[
\Div\nabla \po(\eta_r*\nabla d^\alpha)(\tau,x) )
\ =\ \int_{B(0,r)} \eta_r(z) D^2\po((\eta_r*\nabla d^\alpha(\tau,\cdot))(x))
:D^2 d^\alpha(\tau,x-z)\,dz\,.
\]
The difference in~\eqref{errormollif} is therefore
\[
\int_{B(0,r)} \eta_r(z) (D^2\po(\nabla d^\alpha(\tau,x-z)) - 
D^2\po((\eta_r*\nabla d^\alpha(\tau,\cdot))(x))) :D^2 d^\alpha(\tau,x-z)\,.
\]
Now, since $D^2\po$ is at least continuous (uniformly in
$\{ \po(\xi)\ge  1/2\}$), $\po(\nabla d^\alpha)=1$ a.e.~in $W$,
while $D^2d^\alpha$ is globally bounded (and $\nabla d^\alpha$ uniformly Lipschitz),
this difference can be made arbitrarily small as $r\to 0$, and we
actually deduce that, in such a case, $E'(t)=\{d^\alpha_r\le 0\}$ is a
strict $C^2$-superflow starting from $E'_0=\{\dpE\le \beta \}$,
which satisfies the thesis of the Lemma.
\end{proof}

\begin{remark}\rm
The uniqueness result holds in any dimension $d\ge 2$, with exactly the same proof. It is also not necessary to assume that $G_1$ or $G_2$ vanishes.
\end{remark}

\section{General anisotropies}\label{seccrystal}

An important feature of  Theorem~\ref{thexistsmooth} is that the
existence time, as well as the neighborhood where $\dpE$ is $C^{1,1}$,
are both independent on the anisotropy, and only depend on the
radius $R$ for which $E_0$ satisfies the $R\Wp$-condition.
This allows us to extend the existence result to general anisotropies,
by the approximation argument given in Lemma \ref{lemmapproxcrystal}.

\begin{theorem}\label{teoexistgen} 
Assume $G_1=0$ or $G_2=0$, and let $(\p,\po)$ be an arbitrary
anisotropy. Let $E_0\subset \R^2$ an initial set
with compact boundary, satisfying the $R\Wp$-condition for some $R>0$. 
Then, there exist $T>0$, 
and a $\p$-regular flow $E(t)$ defined on $[0,T]$ 
and starting from $E_0$.

More precisely, 
there exist $R'>0$ and a neighborhood $U$
of $\bigcup_{0\le t\le T} \partial E(t)$ in $\R^2$
such that the sets $E(t)$ satisfy the $R'W_\p$-condition for all $t\in [0,T]$,
the $\p$-signed distance function
$\dpE(t,x)$ from $\partial E(t)$ belongs to $C^0([0,T];\Lip(U))$,
$(\dpE-G)\in \Lip([0,T]\times U)$ and
\begin{equation}\label{nonsmoothevol}
\left|\frac{\partial (\dpE-G)}{\partial t}(t,x)-\Div z(t,x)
\right|\ \le\ \lambda|\dpE(t,x)|
%
\end{equation}
for a.e.~$(t,x)\in [0,T]\times U$,
where $\lambda$ is a positive constant and $z\in L^\infty([0,T]\times U;\R^2)$ is such that 
$z\in\partial\po(\nabla\dpE)$ a.e. in $[0,T]\times U$.
The time $T$, the radius $R'$, and the constant $\lambda$, only depend
on $R$ and $G$.
\end{theorem}

\begin{remark}\textup{
Comparison and uniqueness for such flows has been shown in~\cite{BeNo:99,BCCN-vol,CN-compar}, although the most general result in these references
only covers the case of a time-dependent,
Lipschitz continuous forcing term $G(t)=G(0)+\int_0^t c(s)\,ds$, with $c\in L^\infty(0,+\infty)$.
}\end{remark}

\begin{proof} 
Let $\e>0$ and consider smooth and elliptic anisotropies
$(\p_\e,\po_\e)$, with $\p_\e\ge \p$, converging to $(\p,\po)$ locally uniformly as $\e\to 0$.
By the approximation result in Lemma~\ref{lemmapproxcrystal}, we can find a sequence
of sets $E_\e$ which satisfy the $R\Wpe$-condition,
and such that $\partial E_\e\to \partial E$ in the Hausdorff sense. 
For each $\e$ we consider the evolution $E^\e(t)$ given by Theorem~\ref{thexistsmooth}, 
with $0\le t\le T^\e$. 
Since the times $T^\e$ and the width of the neighborhoods $U^\e$
depend only on $R$ and $G$, up to extracting a subsequence 
we can assume that $\lim_\e T^\e=T$ for some $T>0$, and
there exists a neighborhood $U$ of $\partial E_0$ such that $\R^d\setminus U^\e$
converges to $\R^d\setminus U$ in the Hausdorff sense, as $\e\to 0$.
Possibly reducing $T$ and the width of $U$ we can then assume that $T^\e=T$ and $U^\e=U$ for all $\e>0$.

Letting $W:=[0,T]\times U$,
and
$z_\e(t,x)=\nabla \po_\e (\nabla d_{\p_\e}^{E_\e}(t,x))$, from \eqref{smoothevol} we get
\begin{equation}\label{smoothevoleps}
\left|\frac{\partial (d_{\p_\e}^{E_\e}-G)}{\partial t}(t,x)-\Div z_\e(t,x)
\right|\ \le\ \lambda|d_{\p_\e}^{E_\e}(t,x)|
\end{equation}
for a.e. $(t,x)\in W$,
where the constant $\lambda$ depends only on $R$ and $G$.

As $d_{\p_\e}^{E_\e}-G$ are uniformly Lipschitz in $(t,x)$, 
up to a subsequence we can assume that the functions $d_{\p_\e}^{E_\e}$
converge uniformly in any compact subset of $W$
to a function $\dpE$, such that for all $t\in [0,T]$ $\dpE(t,\cdot)$  
is the signed $\p$-distance function to the boundary of $E(t):=\{x:\,\dpE(t,x)\le 0\}$. 
Moreover, $E(0)=E_0$, $E(t)$ is the
Hausdorff limit of $E_{\e}(t)$ for each $t\in [0,T]$,
and satisfies the  $R'\Wp$-condition,
with $R'=\lim_\e R'_\e$.

Up to a subsequence we can also assume that there exists $z\in L^\infty(W)$
with $z_{\e}(t,x)\wsto z(t,x)$, 
$\Div z_{\e}\wsto \Div z$ and
$\partial_t (d^{\p_{\e}}_{E_{\e}}-G)\wsto \partial_t (\dpE-G)$ in $L^\infty(W)$, so that
\eqref{nonsmoothevol} holds a.e.~in $W$.

It remains to check that
$z(t,x)\in \partial \po(\nabla \dpE(t,x))$. Since by
construction $z(t,x)\in \Wp$ for a.e. $(t,x)\in W$, it is enough to show that
\begin{equation}\label{zwp}
z\cdot \nabla \dpE\ =\ \po(\nabla \dpE)\ =\ 1
\end{equation}
a.e.~in $W$. Recalling that $z_{\e}\cdot\nabla d_{\p_\e}^{E_\e}
= \po_\e (\nabla d_{\p_\e}^{E_\e})=1$ a.e.~in $W$ and letting $\psi\in C_c^\infty(W)$, 
we have
\[
\int_W \psi\,dx dt
\,=\, \int_W \psi \left(z_{\e}\cdot\nabla d_{\p_{\e}}^{E_{\e}}\right)\,dx dt
\, =\, -\int_W d_{\p_{\e}}^{E_{\e}} \left(z_{\e}\cdot\nabla \psi  
+ \psi \Div z_{\e}\right)\,dx dt.
\]
Passing to the limit in the righ-hand side we then get
$$
\int_W \psi\,dx dt
\,=\, 
-\int_W \dpE \left(z\cdot\nabla \psi  
+ \psi \Div z\right)\,dx dt\,=\,\int_W \psi\left(z\cdot \nabla\dpE
\right)\,dxdt
$$
which gives~\eqref{zwp}.
\end{proof}



\begin{thebibliography}{10}

\bibitem{AT}
F.~J. Almgren and J.~E. Taylor.
\newblock Flat flow is motion by crystalline curvature for curves
with crystalline energies.
\newblock {\em J. Diff. Geom.}, 42:1--22, 1995.

\bibitem{ATW} F.~J. Almgren, J.~E. Taylor and L.-H. Wang.
\newblock Curvature-driven flows: a variational approach
\newblock {\em SIAM J. Control Optim.}, 31(2):387--438, 1993.

\bibitem{ACN} L. Almeida, A. Chambolle and M. Novaga.
\newblock Mean curvature flow with obstacles.
\newblock {\em Annales IHP - Analyse Nonlineaire}, to appear.

\bibitem{AmFuPa:00}
L.~Ambrosio, N.~Fusco, and D.~Pallara.
\newblock {\em Functions of Bounded Variation and Free Discontinuity Problems}.
\newblock Oxford Mathematical Monographs, 2000.

\bibitem{An:90}
S. Angenent.
\newblock Parabolic equations for curves on surfaces I. Curves with p-integrable curvature.
\newblock {\em Ann. of Math. (2)}, 132(3):451--483, 1990.

\bibitem{BCCN} G. Bellettini, V. Caselles, A. Chambolle and M. Novaga.
\newblock Crystalline mean curvature flow of convex sets.
\newblock {\em Arch. Rat. Mech. Anal.}, 179(1):109--152, 2006.

\bibitem{BCCN-vol} G. Bellettini, V. Caselles, A. Chambolle and M. Novaga.
\newblock The volume preserving crystalline mean curvature flow of convex
sets in $\R^N$.
\newblock {\em Journal de math{\'e}matiques pures et appliqu{\'e}es},
92(5):499--527, 2009.

\bibitem{BGN:00}
G. Bellettini, R. Goglione and M. Novaga,
\newblock Approximation to driven motion by crystalline curvature
 in two dimensions,
\newblock {\em Adv. Math. Sci. and Appl.}, 10:467-493, 2000.

\bibitem{BeNoPa1}
G. Bellettini, M. Novaga and M. Paolini.
\newblock On a crystalline variational problem I. First variation and
global $L\sp\infty$ regularity.
\newblock {\em Arch. Ration. Mech. Anal.} 157(3),
165--191, 2001.

\bibitem{BeNoPa2} G. Bellettini, M. Novaga, and M. Paolini.
\newblock On a crystalline variational problem. II. $BV$ regularity and structure of minimizers on facets.
\newblock {\em Arch. Ration. Mech.
Anal.} 157(3), 193-217, 2001.

\bibitem{BeNo:99}
G.~Bellettini and M.~Novaga.
\newblock Approximation and comparison for non-smooth anisotropic
motion by mean curvature in $\R^N$.
\newblock {\it Math. Mod. Meth. Appl. Sc.\/},
{\bf 10} (2000), 1--10.




\bibitem{CaHo:72}
J.W. Cahn and D.W. Hoffman.
\newblock A vector thermodynamics for anisotropic interfaces.
1. {F}undamentals and applications to plane surface junctions.
\newblock {\em Surface Sci.}, 31:368--388, 1972.



\bibitem{AC-MCM}
A. Chambolle.
\newblock An algorithm for mean curvature motion.
\newblock {\em Interfaces Free Bound.}, 6:195--218, 2004. 

\bibitem{CN-compar}
A. Chambolle and M. Novaga.
\newblock Convergence of an algorithm for the anisotropic and crystalline mean
  curvature flow.
\newblock {\em SIAM J. Math. Anal.}, 37(6):1978--1987 (electronic), 2006.

\bibitem{mcmf}
A. Chambolle and M. Novaga. 
\newblock Implicit time discretization of the mean curvature flow with a discontinuous forcing term.
\newblock {\em Interfaces Free Bound.}, 10:283--300, 2008.

\bibitem{UsersGuide}
M.~G. Crandall, H. Ishii and P.-L. Lions.
\newblock User's guide to viscosity solutions of second order partial
differential equations.
\newblock {\em Bull. Am. Math. Soc.}, 27:1--67, 1992.

\bibitem{CGG}
Y.~G. Chen, Y. Giga and S. Goto.
\newblock Uniqueness and existence of viscosity solutions of generalized mean curvature flow equations.
\newblock {J. Differential Geom.}, 33(3):749-786, 1991. 

\bibitem{DLN}
N. Dirr, S. Luckhaus, M. Novaga.
\newblock A stochastic selection principle in case of fattening for curvature flow.
\newblock {\em Calc. Var. PDE}, 13(4):405--425, 2001. 

\bibitem{Evans}
L.~C. Evans
\newblock Convergence of an algorithm for mean curvature motion.
\newblock {\em Indiana Univ. Math. J.}, 42:533--557, 1993.

\bibitem{EvGa:92}
L.~C. Evans and R.~F. Gariepy.
\newblock {\em Measure Theory and Fine Properties of Functions}.
\newblock Studies in Advanced Math., CRC Press, Ann Harbor, 1992.

\bibitem{EvSp:II}
L.~C. Evans and J. Spruck.
\newblock Motion of level sets by mean curvature. II.
\newblock {\em Trans. Amer. Math. Soc.}, 330:321--332, 1992.

\bibitem{GH} M. Gage and R. Hamilton.
\newblock The heat equation shrinking convex plane curves. 
\newblock {\em J. Differential Geom.}, 23:69-95, 1986.


\bibitem{GigaGigaRybka}
M.-H. Giga, Y. Giga and P. Rybka,
\newblock A comparison principle for singular diffusion equation
with spatially inhomogeneous driving force for graphs. 
\newblock (preprint \# 981, Hokkaido University, 2011).

\bibitem{GigaGorkaRybka}
Y. Giga, P. G\'orka and P. Rybka,
\newblock Evolution of regular bent rectangles by the driven crystalline
curvature flow in the plane with a non-uniform forcing term.
\newblock (preprint \# 993, Hokkaido University, 2011).

\bibitem{GiGu:96}
Y. Giga and M.~E. Gurtin.
\newblock A comparison theorem for crystalline evolutions in the plane.
\newblock {\em Quarterly of Applied Mathematics}, LIV:727--737, 1996.

\bibitem{LionsSouga}
P.-L. Lions and P.~E. Souganidis.
\newblock Fully nonlinear stochastic partial differential equations.
\newblock {\em C.R. Acad. Sci. - Series I - Mathematics} 326(9):1085--1092, 1998.

\bibitem{LuckhausSturz}
S. Luckhaus and Sturzenecker,
\newblock Implicit time discretization for the mean curvature
flow equation.
\newblock {\em Calc. Var.} {\bf 3} (1995), 253-271.

\bibitem{Ta:78}
J.~E. Taylor.
\newblock Crystalline variational problems.
\newblock {\em Bull. Amer. Math. Soc. (N.S.)}, 84:568--588, 1978.

\bibitem{Ta:93}
J.~E. Taylor,
\newblock Motion of curves by crystalline curvature, including triple
junctions and boundary points,
\newblock {\em Differential geometry:
partial differential equations on manifolds (Los Angeles, CA, 1990)}, 417-438, Proc. Sympos. Pure Math., 54, Part 1, Amer. Math. Soc., Providence, 1993
\end{thebibliography}
\end{document}